\documentclass{amsart}
\usepackage{amsmath}
\usepackage{amsfonts}
\usepackage{graphics}
\usepackage{epsfig}
\usepackage{amssymb}
\usepackage{amscd}
\usepackage[all]{xy}
\usepackage{latexsym}
\usepackage{graphicx}
\usepackage{multirow}
\usepackage{geometry}

\theoremstyle{remark}
\newtheorem{lem}{\bf Lemma}[section]
\newtheorem{defn}{\bf Definition}[section]
\newtheorem{thm}{\bf Theorem}[section]
\newtheorem{cor}{\bf Corollary}[section]
\newtheorem{ex}{\bf Example}[section]

\newtheorem*{rem}{\bf Remark}

\newtheorem{prop}{\bf Proposition}[section]

\numberwithin{equation}{section} \numberwithin{figure}{section}

\renewcommand*{\to}{\rightarrow}
\renewcommand*{\bar}[1]{\overline{#1}}

\newcommand{\mb}[1]{\mathbb{#1}} 

\newcommand{\mk}[1]{\mathfrak{#1}}

\newcommand{\diag}{\operatorname{diag}}

\newcommand{\Jac}{\operatorname{Jac}}

\newcommand{\dsp}{\displaystyle}
\newcommand{\mf}[1]{\mathbf{#1}}

\title[Mean field Equations on hyperelliptic curves]{An algebraic construction of a solution to the mean field equations on hyperelliptic Curves and its adiabatic limit}

\author{Jia-Ming (Frank) Liou}
\address{Department of Mathematics\\
National Cheng Kung University, Taiwan\\
fjmliou@mail.ncku.edu.tw}
\address{NCTS, Mathematics}

\author{Chih-Chung Liu}
\address{Department of Mathematics\\
National Cheng Kung University, Taiwan\\
cliu@mail.ncku.edu.tw}

\begin{document}
\setcounter{section}{0}
\maketitle

\begin{abstract}\large
In this paper, we give an algebraic construction of the solution to the following mean field equation
$$
\Delta \psi+e^{\psi}=4\pi\sum_{i=1}^{2g+2}\delta_{P_{i}},
$$
on a genus $g\geq 2$ hyperelliptic curve $(X,ds^{2})$ where $ds^{2}$ is a canonical metric on $X$ and $\{P_{1},\cdots,P_{2g+2}\}$ is the set of Weierstrass points on $X.$ 
\end{abstract}
\large

\section{Introduction}
Let $f(x)$ be a complex polynomial in $x$ with  $2g+2$ distinct complex roots $\{e_{1},\cdots,e_{2g+2}\}.$ The affine plane curve $C_{0}=\{(x,y)\in\mb C^{2}:y^{2}=f(x)\}$ defines a noncompact Riemann surface with respect to the complex analytic topology on $\mb C^{2}.$ To compactifiy $C_{0}$ in the category of Riemann surfaces, we introduce another smooth affine plane curve $C_{0}'.$ Let $g(z)$ be the complex polynomial $g(z)=\prod_{i=1}^{2g+2}(1-e_{i}z)$ and $C_{0}'$ be the smooth affine plane curve defined by $w^{2}=g(z),$ i.e. $C_{0}'=\{(z,w)\in\mb C^{2}:w^{2}=g(z)\}.$ Let $U_{0}$ be the open subset of $C_{0}$ consisting of points $(x,y)$ so that $x\neq 0$ and $U_{0}'$ be the open subset of $C_{0}'$ consisting of points $(z,w)$ such that $z\neq 0.$ The map $\varphi:U_{0}\to U_{0}'$ defined by $\varphi(x,y)=(1/x,y/x^{g+1})$ is an isomorphism of Riemann surfaces. It is well known that the gluing $C_{0}\cup_{\varphi}C_{0}'$ of $C_{0}$ and $C_{0}'$ along $\varphi$ is a connected compact Riemann surface of genus $g;$ see \cite{Miran} or \cite{Mum}. The compact Riemann surface $C_{0}\cup_{\varphi}C_{0}'$ is called the hyperelliptic curve of genus $g$ defined by $y^{2}=f(x)$ and is denoted by $X$ in this paper. The holomorphic map $\pi:X\to\mb P^{1}$ defined by
$$\pi(P)=
\begin{cases}
(x(P):1) &\mbox{if $P\in C_{0},$}\\
(1:z(P)) &\mbox{if $P\in C_{0}'$}
\end{cases}$$
is a degree two ramified covering map of $\mb P^{1}$ where $(z_{0}:z_{1})$ is the homogeneous coordinate on $\mb P^{1}.$ The Weierstrass points of $X$ are the $2g+2$ ramification points $\{P_{1},\cdots,P_{2g+2}\}$ of $\pi$ such that $(x(P_{k}),y(P_{k}))=(e_{k},0)$ for $1\leq k\leq 2g+2.$

The space $H^{0}(X,\Omega_{X}^{1})$ of holomorphic one forms on $X$ has a simple basis of the form $\{x^{i-1}dx/y:1\leq i\leq g\}$ and the integral homology group $H_{1}(X)$ of $X$ has a (symplectic) $\mb Z$-basis $\{a_{i},b_{j}:1\leq i,j\leq g\}$ such that
$\int_{a_{j}}\omega_{i}=\delta_{ij}$ for $1\leq i,j\leq g.$ Denote $\tau_{ij}=\int_{b_{j}}\omega_{i}$ for $1\leq i,j\leq g$ and $\tau$ be the complex $g\times g$ matrix $[\tau_{ij}]_{i,j=1}^{g}.$ Then $\tau$ is a symmetric matrix with positive definite imaginary part. Let $\Lambda_{\tau}$ be the lattice in $\mb C^{g}$ generated by the column vectors of the $g\times 2g$ matrix $\Omega=[I_{g},\tau].$ Let $\Omega_{i}$ be the $i$-th column vector of $\Omega$ and $\{dx_{1},\cdots,dx_{2g}\}$ be the real basis dual to $\{\Omega_{i}:1\leq i\leq 2g\}.$ The complex torus $\Jac(X)=\mb C^{g}/\Lambda_{\tau}$ together with the class $[\omega]$ where $\omega=\sum_{i=1}^{g}dx_{i}\wedge dx_{g+i}$ is a principally polarized abelian variety called the Jacobian variety of $X.$ Fixing a point $P_{0}$ on $X,$ we define a holomorphic map
$$\mu:X\to \Jac(X),\quad \mu(P)=\left(\int_{P_{0}}^{P}\frac{dx}{y},\cdots,\int_{P_{0}}^{P}\frac{x^{g-1}dx}{y}\right)\mod\Lambda.$$
Let $(z_{1},\cdots,z_{g})$ be the standard holomorphic coordinate on $\Jac(X)$ and $d\widetilde{s}^{2}_{H}$ be the flat hermtian metric $\sum_{i,j=1}^{g}h_{ij}dz^{i}\otimes d\bar{z}^{j}$ on $\Jac(X)$ where $H=[h_{ij}]$ is a $g\times g$ positive definite hermitian matrix. The canonical metric $ds_{H}^{2}$ on $X$ is defined by $ds_{H}^{2}=\mu^{*}d\widetilde{s}_{H}^{2}$ and has the form
$$ds_{H}^{2}=\frac{1}{|y^{2}|}\sum_{i,j=1}^{g}h_{ij}x^{i-1}\bar{x}^{j-1}.$$
Let $\Delta_{H}$ be the Laplace operator associated with the metric $ds_{H}^{2}.$ In this paper, we study the following mean field equation
\begin{equation}\label{MFE}
\Delta_{H}\psi+e^{\psi}=4\pi\sum_{i=1}^{2g+2}\delta_{P_{i}},
\end{equation}
where $\{P_{1},\cdots,P_{2g+2}\}$ is the set of all Weierstrass points on $X$ and $\delta_{P}:C^{\infty}(X)\to\mb C$ is the Dirac measure centered at $P$ for $P\in X.$

In \cite{FL}, we discovered that when $X$ has genus two, the Gaussian curvature function $K$ of a canonical metric determines the solution $\psi$ to (\ref{MFE}). This paper is a continuation of \cite{FL}; we give an algebraic construction of the solution to (\ref{MFE}) involving the study of solutions to the formal nonlinear ordinary differential equation
\begin{equation}\label{ode1}
(tQ''(t)+Q'(t))Q(t)-t(Q'(t))^{2}=S(t)Q(t),
\end{equation}
and the study of solutions to the formal nonlinear partial differential equation
\begin{equation}\label{1}
u\frac{\partial^{2}u}{\partial x\partial y}-\frac{\partial u}{\partial x}\frac{\partial u}{\partial y}=\sigma u.
\end{equation}
Here $S(t)$ is a complex formal power series and $\sigma$ is a complex polynomial in $x,y.$ We call $S$ the data for (\ref{ode1}) and $\sigma$ the data for (\ref{1}) respectively. In Section \ref{no}, we define a sequence of polynomials to solve (\ref{ode1}) and give a necessary and sufficient condition for (\ref{ode1}) possesing a polynomial solution. In Section \ref{np}, we show that the existence of solutions to (\ref{1}) is equivalent to the nonemptyness of certain (ind) affine algebraic set. Solutions to (\ref{ode1}) would allow us to construct solutions to (\ref{1}) for certain type of polynomials $\sigma.$ In Section \ref{smfe}, using the method developed in Section \ref{no} and Section \ref{np}, we give a construction of the solution to (\ref{MFE}) and the closed form of the solution to (\ref{MFE}) when $H$ is a diagonal matrix with positive diagonals. In section \ref{al}, we introduce a real positive parameter $\gamma$ into \eqref{MFE} and generalize the solutions constructed in previous sections. The parameter arises from a rescaling of canonical metric by $\gamma$ and we discuss the adiabatic limit of solutions as $\gamma \to 0$.

\vspace{1cm}

{\bf Acknowledgements}
The authors would like to thank Prof. Chin-Lung Wang (NTU), Prof. Yu-Chen Shu (NCKU), and Prof. Huailiang Chang (HKUST) for many useful discussions. The draft of this paper was completed during the visit to Prof. Huailiang Chang at the Hong Kong University of Science and Technology. The first author was partially supported by MOST Grant 105-2115-M-006-014 and by NCTS. The second author was partially supported by MOST Grant 105-2115-M-006-012.

\section{A formal nonlinear Ordinary Differential Equation}\label{no}

Let $\{\lambda_{i}:i\geq 0\}$ be an infinite sequence of variables and $\mb K[\Lambda]$ be the ring of polynomials in $\{\lambda_{i}:i\geq 0\}$ over a field $\mb K.$ We define a sequence of polynomials $\{f_{\lambda}^{i}:i\geq 1\}$ in $\mb Q[\Lambda][t]$ by $f_{\lambda}^{1}(t)=\lambda_{0}$ and
\begin{equation}\label{dr1}
f_{\lambda}^{k+1}(t)=\frac{\lambda_{k}}{(k+1)^{2}}+\frac{t}{(k+1)^{2}}\sum_{i=0}^{k-1}(\lambda_{i}-(i+1)(2i+1-k)f_{\lambda}^{i+1}(t))f_{\lambda}^{k-i}(t),\quad k\geq 2.
\end{equation}
By definition,
$$f_{\lambda}^{2}(t)=\frac{\lambda_{0}^{2}t+\lambda_{1}}{4},\quad f_{\lambda}^{3}(t)=\frac{\lambda_{0}\lambda_{1}t+\lambda_{2}}{9},\quad f_{\lambda}^{4}(t)=-\lambda_{0}^{2}\lambda_{1}t^{2}+(3\lambda_{1}^{2}+8\lambda_{0}\lambda_{2})t+\frac{\lambda_{3}}{16}$$
By induction, the degree of $f_{\lambda}^{i}(t)$ in $t$ is $i-2$ for $i\geq 3$ and the $t^{0}$ term of $f_{\lambda}^{i}(t)$ is $\lambda_{i-1}/(i-1)^{2}$ for $i\geq 2.$ Let us denote $f_{\lambda}^{i}$ by
$$f_{\lambda}^{i}(t)=\sum_{j=0}^{i-2}\beta_{ij}(\lambda)t^{j},\quad \beta_{ij}(\lambda)\in\mb Q[\Lambda].$$
If $S(t)=\sum_{i=0}^{\infty}s_{i}t^{i}$ is a complex formal power series, we set $f_{S}^{1}(t)=s_{0}$ and
$$f_{S}^{i}(t)=\sum_{j=0}^{i-2}\beta_{ij}(s_{0},s_{1},\cdots)t^{j}$$
Notice that when $S(t)$ is a polynomial of degree at most $n,$ then $f_{S}^{i}(t)$ is divisible by $t$ for all $i\geq n+2.$

\begin{lem}\label{lemp}
Let $f(t)$ and $g(t)$ be complex power series such that $g(0)\neq 0$ and $f(t)g(t)$ is divisible by $t^{m}$ for some $m\geq 1.$ Then $f(t)$ is divisible by $t^{m}.$
\end{lem}
\begin{proof}
We assume that $f(t)g(t)=t^{m}h(t)$ for some $h(t)\in\mb C[[t]].$ Then $f(0)=0.$ Taking the (formal) derivatives of the equation $f(t)g(t)=t^{m}h(t)$ with respect to $t$ and using the fact that $g(0)\neq 0,$ we prove by induction that $f^{(i)}(0)=0$ for $1\leq i\leq m-1.$ This implies that $f(t)=t^{m}f_{1}(t)$ with $f_{1}\in\mb C[[t]].$
\end{proof}

Assume that $Q(t)$ is a solution to (\ref{ode1}) and $Q(t)$ is divisible by $t^{m}$ but not by $t^{m+1}.$ Define $Q_{1}(t)\in\mb C[[t]]$ such that $Q(t)=t^{m}Q_{1}(t).$ ($Q_{1}(t)$ is defined since $\mb C[[t]]$ is a unique factorization domain.) Then $Q_{1}(0)\neq 0.$ By an elementary computation,
$$t^{m}\left(\left(tQ_{1}''(t)+Q_{1}'(t)Q_{1}(t)\right)-t(Q_{1}'(t))^{2}\right)=S(t)Q_{1}(t).$$
By Lemma \ref{lemp}, $S(t)$ is divisible by $t^{m}.$ Define $S_{1}(t)$ by $S(t)=t^{m}S_{1}(t).$ Then
$$\left(tQ_{1}''(t)+Q_{1}'(t)Q_{1}(t)\right)-t(Q_{1}'(t))^{2}=S_{1}(t)Q_{1}(t).$$
This shows that $Q_{1}(t)$ is a solution to (\ref{ode1}) with the data $S_{1}(t)$ and with the initial condition $Q_{1}(0)\neq 0.$ Owing to this observation, it suffices to consider the solutions $Q(t)$ to (\ref{ode1}) for a given data $S(t)$ under the assumption $Q(0)\neq 0.$

\begin{prop}\label{odes}
Let $S(t)=\sum_{i=0}^{\infty}s_{i}t^{i}$ be a complex formal power series and $a$ be a nonzero complex number. A formal power series $Q(t)=\sum_{i=1}^{\infty}q_{i}t^{i}$ with $Q(0)=1/a$ solves (\ref{ode1}) for the data $S(t)$ if and only if $q_{i}=f_{S}^{i}(a)$ for $i\geq 1.$
\end{prop}
\begin{proof}
After some basic computation, we know
\begin{align*}
(tQ''(t)+Q'(t))Q(t)-t(Q'(t))^{2}&=\sum_{k=0}^{\infty}\left(\sum_{i=0}^{k}(i+1)(2i+1-k)q_{i+1}q_{k-i}\right)t^{k},\\
S(t)Q(t) &=\sum_{k=0}^{\infty}\left(\sum_{i=0}^{k}s_{i}q_{k-i}\right)t^{k},
\end{align*}
If $Q(t)$ is a solution to (\ref{ode1}), then
\begin{equation}\label{r1}
\sum_{i=0}^{k}(i+1)(2i+1-k)q_{i+1}q_{k-i}=\sum_{i=0}^{k}s_{i}q_{k-i},\quad k\geq 0.
\end{equation}
Hence $q_{0}q_{1}=q_{0}s_{0}.$ Since $q_{0}\neq 0,$ $q_{1}=s_{0}.$ Then $q_{1}=f_{S}^{1}(a)$ holds. Furthermore, (\ref{r1}) can be rewritten as:
\begin{equation}\label{r2}
(k+1)^{2}q_{k+1}q_{0}=s_{k}q_{0}+\sum_{i=0}^{k-1}(s_{i}-(i+1)(2i+1-k)q_{i+1})q_{k-i},\quad k\geq 1
\end{equation}
which implies that (by $q_{0}=1/a$)
$$q_{k+1}=\frac{s_{k}}{(k+1)^{2}}+\frac{a}{(k+1)^{2}}\sum_{i=0}^{k-1}(s_{i}-(i+1)(2i+1-k)q_{i+1})q_{k-i},\quad k\geq 1.$$
By (\ref{dr1}) and induction, $q_{k+1}=f_{S}^{k+1}(a)$ for $k\geq 1.$ For the converse, since $q_{i}=f_{S}^{i}(a),$ $(q_{i})$ satisfies (\ref{r1}). Define $Q(t)=\sum_{i=0}^{\infty}q_{i}t^{i}.$ By (\ref{r1}), $Q(t)$ satisfies (\ref{ode1}). We complete the proof of our assertion.
\end{proof}
This proposition implies that the solution to (\ref{ode1}) is uniquely determined by the initial condition $Q(0)=a$ with $a\neq 0$ and the solution can be constructed by the numbers $f_{S}^{i}(a)$ for $i\geq 1.$

When $S(t)$ is a polynomial of degree $m,$ we would like to find the necessary and the sufficient condition for (\ref{ode1}) possessing polynomial solutions.

\begin{lem}\label{ld}
Let $S(t)$ be a polynomial of degree $m.$ If the solution $Q(t)$ to (\ref{ode1}) for the data $S(t)$ is a polynomial of degree $n$ in $t,$ then $n\geq m+2.$
\end{lem}
\begin{proof}
The polynomial $(tQ''(t)+Q'(t))Q(t)-t(Q'(t))^{2}$ has degree at most $2n-2$ while the degree of $S(t)Q(t)$ is $n+m.$ Hence $n+m\leq 2n-2$ implies that $n\geq m+2.$
\end{proof}

\begin{prop}\label{propf}
Let $S(t)$ be a polynomial of degree $m.$ The solution $Q(t)$ to (\ref{ode1}) for the data $S(t)$ constructed in Proposition (\ref{odes}) is a polynomial in $t$ if and only if there exists $N\in\mb N$ with $N\geq m+2$ such that $q_{0}^{-1}$ is the common root of the polynomials $\{f_{S}^{i}:N+1\leq i\leq 2N-1\}.$ Here $q_{0}=Q(0).$
\end{prop}
\begin{proof}
Suppose that $Q(t)=\sum_{i=0}^{\infty}q_{i}t^{i}$ is a polynomial of degree $n.$ Then $q_{i}=0$ for all $i\geq n+1.$ We choose $N=n.$ By Lemma \ref{ld}, $N\geq m+2.$ Since $q_{i}=f_{S}^{i}(q_{0}^{-1}),$ $q_{0}^{-1}$ is a root of $f_{S}^{i}$ for all $i\geq N+1$ and hence $f_{S}^{i}(q_{0}^{-1})=0$ for all $i\geq N+1.$ Therefore $q_{0}^{-1}$ is a root of $f_{S}^{i}(t)$ for $i\geq N+1$ and thus for $N+1\leq i\leq 2N-1.$

Let us prove the converse. Assume that $q_{0}^{-1}$ is a common root of $f_{S}^{i}(t)$ for $N+1\leq i\leq 2N-1.$ Let us prove the statement $q_{2N-1+j}=0$ for $j\geq 1$ by induction on $j.$ For $j=1,$
$$q_{2N}=\frac{1}{(2N)^{2}q_{0}}\sum_{i=0}^{2N-2}(s_{i}-(i+1)(2N-2i+2)q_{i+1})q_{2N-1-i}.$$
For $0\leq i\leq N-2,$ $N+1\leq 2N-1-i\leq 2N-1.$ Hence $q_{2N-1-i}=0$ for $0\leq i\leq N-2.$ For $i\geq N-1,$ $i\geq m+1$ and hence $s_{i}=0$ for $i\geq N-1.$ For $N\leq i\leq 2N-2,$ $N+1\leq i+1\leq 2N-1.$ Hence $q_{i+1}=0$ for $N\leq i\leq 2N-2$ by assumption. Notice that when $i=N-1,$ $2i-2N+2=0.$ We conclude that $q_{2N}=0.$ This proves that the statement holds for $j=1.$ We assume that the statement is true for $0\leq j\leq l.$ If $j=l+1,$ $2N-1+j\geq N\geq m+2$ and hence $s_{2N-1+j}=0$ which implies that
$$q_{2N+l}=\frac{1}{(2N+l)^{2}q_{0}}\sum_{i=0}^{2N+l-2}(s_{i}-(i+1)(2i+1-k)q_{i+1})q_{2N+l-1-i}.$$
For $0\leq i\leq N-1,$ $N+1\leq N+l<2N-1+l-i\leq 2N-1+l-i\leq 2N-1+l.$ By induction hypothesis and the assumption, we obtain that $q_{2N-1+l-i}=0$ for $0\leq i\leq N-1.$ For $N\leq i\leq 2N+l-2,$ $N+1\leq i+1\leq 2N-1+l$ and hence $s_{i}=q_{i+1}=0.$ We conclude that $q_{2N+l}=0.$ We prove that $q_{2N-1+j}=0$ holds for $j=l+1.$ By mathematical induction, $q_{2N-1+j}=0$ for all $j\geq 1.$ Combining with the assumption, one has $q_{i}=0$ for all $i\geq N+1.$ Therefore $Q(t)$ is a polynomial.
\end{proof}

In fact, we can prove that:

\begin{lem}
Let $x_{0},\cdots,x_{m}$ be a set of formal variables for $m\geq 1.$ For each $i,$ define a polynomial over $\mb Q$ in $x_{0},\cdots,x_{m},t$ by
\begin{equation}\label{poly}
F^{i}(x_{0},\cdots,x_{m},t)=f_{x_{0}+x_{1}t+\cdots+x_{m}t^{m}}^{i}(t),
\end{equation}
where $f_{S}^{i}$ is the polynomial defined in (\ref{dr1}). Let $N$ be a natural number so that $N\geq m+2.$ The set of polynomials $\{F^{i}:N+1\leq i\leq 2N-1\}$ is divisible by $G(x_{0},\cdots,x_{g-1},t)\in\mb Q[x_{0},\cdots,x_{g-1},t]$ if and only if $\{F^{i}:i\geq N+1\}$ is divisible by $G.$
\end{lem}
\begin{proof}
The proof follows from the recursive relation
$$F^{k+1}=\frac{t}{(k+1)^{2}}\sum_{i=0}^{k-1}(x_{i}-(i+1)(2i+1-k)F^{i+1})F^{k-i},\quad k\geq m+1$$
and is similar as that given in Proposition \ref{propf}. We leave it to the readers.
\end{proof}

\begin{cor}\label{cor1}
Let $S(t)$ be a complex polynomial of degree $m$ and $Q(t)$ is a solution to (\ref{ode1}) for the data $S(t)$ such that $Q(0)^{-1}$ is a common zero of $\{f_{S}^{i}:m+3\leq i\leq 2m+3\}.$ Then $Q(t)$ is a polynomial of degree $m+2.$
\end{cor}
\begin{proof}
By assumption and the Proposition \ref{propf}, $q_{i}=0$ for $i\geq m+3.$ Then $Q(t)$ is a polynomial of degree at most $m+2.$ By Lemma \ref{ld}, the degree of $Q(t)$ is at least $m+2.$ We conclude that $Q(t)$ is a polynomial of degree $m+2.$
\end{proof}

\begin{lem}\label{resc}
For any $i\geq 1,$ and any $S(t)\in\mb C[[t]],$
$$f_{\lambda S}^{i}(t)=\lambda f_{S}^{i}(\lambda t)$$
for any $\lambda\in\mb C.$
\end{lem}
\begin{proof}
When $i=1,$ the statement is obvious. One can also verify that the statement is true for $i=2$ and $3$. Assume that the statement is true for $i=k.$ For $i=k+1,$ we use the recursive relation:
\begin{align*}
f_{\lambda S}^{k+1}(t)&=\frac{\lambda s_{k}}{(k+1)^{2}}+\frac{t}{(k+1)^{2}}\sum_{i=0}^{k-1}(\lambda s_{i}-(i+1)(2i+1-k)f_{\lambda S}^{i+1}(t))f_{\lambda S}^{k-i}(t)\\
&=\frac{\lambda s_{k}}{(k+1)^{2}}+\frac{t}{(k+1)^{2}}\sum_{i=0}^{k-1}(\lambda s_{i}-(i+1)(2i+1-k)\lambda f_{ S}^{i+1}(\lambda t)) \lambda f_{\lambda S}^{k-i}(\lambda t)\\
&=\lambda\left(\frac{s_{k}}{(k+1)^{2}}+\frac{\lambda t}{(k+1)^{2}}\sum_{i=0}^{k-1}(s_{i}-(i+1)(2i+1-k)f_{S}^{i+1}(\lambda t))f_{ S}^{k-i}(\lambda t)\right)\\
&=\lambda f_{S}^{i}(\lambda t).
\end{align*}
\end{proof}

This lemma implies that:
\begin{cor}
Let $S(t)$ be a complex polynomial of degree $m.$ Then (\ref{ode1}) has a polynomial solution for data $S(t)$ if and only if (\ref{ode1}) has polynomial solution for data $\lambda S(t)$ for $\lambda \in\mb C^{*}.$

\end{cor}

Given any complex polynomial $B(t)=\sum_{i=0}^{n}b_{i}t^{i},$ we define a new polynomial $\widetilde{B}(t)$ by
$$\widetilde{B}(t)=t^{\deg B}B(t^{-1})$$
and write $\widetilde{B}(t)=\sum_{i=0}^{n}\widetilde{b}_{i}t^{i}$ where $n=\deg B(t).$ Then $\widetilde{b}_{i}=b_{n-i}$ for $0\leq i\leq n.$

\begin{prop}
Let $S(t)$ be a complex polynomial of degree $m.$ Suppose that (\ref{ode1}) has a polynomial solution $Q(t)$ of degree $n$ for the data $S(t).$ Then $\widetilde{Q}(t)$ solves (\ref{ode1}) for the data $t^{n-m-2}\widetilde{S}(t).$
\end{prop}
\begin{proof}
One uses the chain rules to prove the statement while the calculation is elementary.
\end{proof}

This Proposition implies that

\begin{cor}
Let $S(t)$ be a complex polynomial of degree $m.$ Suppose that (\ref{ode1}) has a polynomial solution $Q(t)$ of degree $m+2$ for the data $S(t).$ Then $\widetilde{Q}(t)$ solves (\ref{ode1}) for the data $\widetilde{S}(t).$
\end{cor}

\section{A formal nonlinear Partial Differential Equation}\label{np}
Let $M_{n}(\mb C)$ be the algebra of $n\times n$ complex matrices. For each $n\geq 1,$ we consider the algebra monomorphism $\psi_{n,n+1}:M_{n}(\mb C)\to M_{n+1}(\mb C)$ defined by
$$\psi_{n,n+1}(A)=
\left[\begin{array}{cc}
A & 0\\
0 & 0
\end{array}
\right].$$
The direct limit of the directed system $\{(M_{n}(\mb C),\psi_{n,m})\}$ is denoted by $M_{\infty}(\mb C)$ where the algebra monomorphism $\psi_{n,m}:M_{n}(\mb C)\to M_{m}(\mb C)$ for $n<m$ is defined by $$\psi_{n,m}=\psi_{m,m-1}\circ\cdots\circ\psi_{n+1,n}.$$
Denote the canonical map $M_{n}(\mb C)\to M_{\infty}(\mb C)$ by $\psi_{n}$ and identify $M_{n}(\mb C)$ with its image in $M_{\infty}(\mb C).$ Then $M_{\infty}(\mb C)$ can be realized as a union $\bigcup_{n=1}^{\infty}M_{n}(\mb C);$ $M_{\infty}(\mb C)$ is an ind-variety over $\mb C.$

By an ind-variety over a field $k,$ we mean that a set $X$ together with a filtration $X_{0}\subset X_{1}\subset X_{2}\subset\cdots$ such that $\bigcup_{n\geq 0}X_{n}=X$ and each $X_{n}$ is a finite dimensional variety over $k$ such that the inclusion $X_{n}\to X_{n+1}$ is a closed embedding. An ind-variety has a natural topology defined as follows. A subset $U$ of $X$ is said to be open if and only if $U\cap X_{n}$ is open in $X_{n}$ for each $n\geq 0.$ The ring of regular functions on $X$ denoted by $k[X]$ is defined to be $k[X]=\varprojlim_{n}k[X_{n}].$ An ind-variety is said to be projective, resp. affine, if each $X_{n}$ is projective, resp. affine. For more details about ind-varieties, see \cite{sha}.

For each $A\in M_{\infty}(\mb C),$ we may write $A=(a_{ij})_{i,j=1}^{\infty}$ with $a_{ij}=0$ for all but finitely many $i,j.$ We associate to $A$ a complex polynomial $\mk p(A)(x,y)$ in $x,y$ by
$$\mk p(A)(x,y)=\sum_{i,j=0}^{\infty}a_{i+1,j+1}x^{i}y^{j}.$$
We obtain a linear monomorphism $\mk p:M_{\infty}(\mb C)\to \mb C[x,y].$ The image of $\mk p$ is denoted by $\mk P_{\infty}[x,y].$ Given $\sigma\in \mk P_{\infty}[x,y],$ we would like to solve for the formal nonlinear differential equation (\ref{1}) in $\mk P_{\infty}[x,y].$ To solve for (\ref{1}) in $\mk P_{\infty}[x,y],$ let us assume that
$$u(x,y)=\sum_{\alpha,\beta=0}^{\infty}a_{\alpha+1,\beta+1}x^{\alpha}y^{\beta}\mbox{ and }\sigma(x,y)=\sum_{i,j=0}^{\infty}c_{i+1,j+1}x^{i}y^{j}.$$ By simple computation,
\begin{align*}
uu_{xy}-u_{x}u_{y}&=\sum_{\alpha,\beta=0}^{\infty}\left(\sum_{i=0}^{\alpha+1}\sum_{j=0}^{\beta+1}i(2j-\beta-1)a_{i+1,j+1}a_{\alpha-i+2,\beta-j+2}\right)x^{\alpha}y^{\beta}\\
\sigma u &=\sum_{\alpha,\beta=0}^{\infty}\left(\sum_{i=0}^{\alpha}\sum_{j=0}^{\beta}a_{i+1,j+1}c_{\alpha-i+1,\beta-j+1}\right)x^{\alpha}y^{\beta}
\end{align*}
Then $u$ solves (\ref{1}) if and only if
$$\sum_{i=0}^{\alpha+1}\sum_{j=0}^{\beta+1}i(2j-\beta-1)a_{i+1,j+1}a_{\alpha-i+2,\beta-j+2}=\sum_{i=0}^{\alpha}\sum_{j=0}^{\beta}a_{i+1,j+1}c_{\alpha-i+1,\beta-j+1}$$
for any $\alpha,\beta\geq 0.$ For each $\alpha,\beta,$ we define
$$\varphi_{\sigma}^{\alpha,\beta}(A)=\sum_{i=0}^{\alpha+1}\sum_{j=0}^{\beta+1}i(2j-\beta-1)a_{i+1,j+1}a_{\alpha-i+2,\beta-j+2}-\sum_{i=0}^{\alpha}\sum_{j=0}^{\beta}a_{i+1,j+1}c_{\alpha-i+1,\beta-j+1}.$$
Then $u=\mk p(A)$ for some $A\in M_{\infty}(\mb C)$ solves (\ref{1}) for data $\sigma$ if and only if $\varphi_{\sigma}^{\alpha,\beta}(A)=0$ for all $\alpha,\beta,$ i.e. $A$ satisfies a family of quadratic polynomials. The subset
$$V_{\sigma}=\{A\in M_{\infty}(\mb C):\varphi_{\sigma}^{\alpha,\beta}(A)=0\}$$
of $M_{\infty}(\mb C)$ is called the ind-affine algebraic variety associated with $\sigma.$ The equation (\ref{1}) has a solution for $\sigma$ if and only if $V_{\sigma}$ is nonempty.

For each $u\in \mk P_{\infty}[x,y],$ we define $M_{xy}u$ by $$(M_{xy}u)(x,y)=(xy)u(x,y).$$ Then $M_{xy}$ defines a linear endomorphism on $\mk P_{\infty}[x,y].$

\begin{lem}
Suppose that $u\in \mk P_{\infty}[x,y]$ is a solution to (\ref{1}) for data $\sigma.$ Then $M_{xy}u$ is a solution to (\ref{1}) for data $M_{xy}\sigma.$
\end{lem}
\begin{proof}
Let $v=M_{xy}u.$ Then $v(x,y)=(xy)u(x,y).$ Hence $v_{x}=yu+(xy)u_{x},$ and $v_{y}=xu+(xy)u_{y},$ and $v_{xy}=u+yu_{y}+xu_{x}+(xy)u_{xy}.$ We discover that
$$vv_{xy}-v_{x}v_{y}=(xy)^{2}(uu_{xy}-u_{x}u_{y})=(xy)^{2}\sigma u=(M_{xy}\sigma) v.$$
This proves our assertion.
\end{proof}

By making use of the fact that $\mb C[x,y]$ is a unique factorization domain, we prove the following fact:

\begin{prop}
Let $v\in \mk P_{\infty}[x,y]$ be a solution to (\ref{1}) for a data $\sigma\in \mk P_{\infty}[x,y].$ Assume that there exists $m\in\mb N$ such that $v$ is divisible by $(xy)^{m}$ but not by $x^{m+1}y^{m}$ and not by $x^{m}y^{m+1}.$ Then $\sigma$ is divisible by $(xy)^{m}.$ Furthermore, if $u\in \mk P_{\infty}[x,y]$ and $\gamma\in \mk P_{\infty}[x,y]$ are polynomials so that $v=M_{xy}^{m}u$ and $\sigma=M_{xy}^{m}\gamma,$ then $u$ is a solution to (\ref{1}) for the data $\gamma.$
\end{prop}
\begin{proof}
Since $v$ is divisible by $(xy)^{m},$ we write $v=M_{xy}^{m}u$ for some $u\in \mk P_{\infty}[x,y].$ We can show that
$$vv_{xy}-v_{x}v_{y}=(xy)^{2m}(uu_{xy}-u_{x}u_{y}).$$
Since $vv_{xy}-v_{x}v_{y}=\sigma v=(xy)^{m}\sigma u,$ we find
$$\sigma u=(xy)^{m}(uu_{xy}-u_{x}u_{y}).$$
Since $v$ is not divisible by $x^{m+1}y^{m}$ and not by $x^{m}y^{m+1},$ $u$ is not divisible by $x$ and $y.$ We see that $\sigma$ is divisible by $(xy)^{m}.$ Let $\sigma=M_{xy}^{m}\gamma$ for $\gamma\in \mk P_{\infty}[x,y].$ Then
$$uu_{xy}-u_{x}u_{y}=\gamma u.$$
This proves our assertion.
\end{proof}

\begin{defn}
A solution $u\in \mk P_{\infty}[x,y]$ to (\ref{1}) for a given data is called a prime solution to (\ref{1}) if $u$ is not divisible by $xy.$
\end{defn}

Let us denote the image of $M_{n}(\mb C)$ in $\mb C[x,y]$ via $\mk p$ by $\mk P_{n}[x,y].$ Then $\mk P_{\infty}[x,y]=\bigcup_{n\geq 1}\mk P_{n}[x,y].$

\begin{lem}
Let $\sigma\in \mk P_{m}[x,y]$ with $\deg \sigma=2m-2.$ If $u\in\mk P_{\infty}[x,y]$ is a solution to (\ref{1}) for the data $\sigma$ of degree $2n-2,$ then $n\geq m+2.$
\end{lem}
\begin{proof}
We observe that the coefficients of $x^{2n-3}y^{2n-3}$ and of $x^{2n-3}y^{2n-4}$ and of $x^{2n-4}y^{2n-3}$ in $uu_{xy}-u_{x}u_{y}$ all vanish. Then $uu_{xy}-u_{x}u_{y}$ is a polynomial of degree at most $4n-8.$ On the other hand, the degree of $\sigma u$ is $2n+2m-4.$ We conclude that $n\geq m+2.$
\end{proof}

Let us write a remark that $V_{\sigma}$ is an ind-affine variety. Given $\sigma\in \mk P_{m}[x,y]$ with degree $2m-2,$ the intersection $V_{\sigma}^{n}=V_{\sigma}\cap M_{n}(\mb C)$ is an affine algebraic subvariety of $M_{n}(\mb C)\cong\mb A^{n^{2}}(\mb C)$ for $n\geq m+2$ and $V_{\sigma}=\bigcup_{n\geq m+2}V_{\sigma}^{n}.$ \\

Let $q\in \mk P_{n}[x,y].$ Formally, we define
$$\widetilde{q}(x,y)=(xy)^{n}q(x^{-1},y^{-1}).$$

\begin{lem}
Let $\sigma\in \mk P_{m}[x,y]$ be given with $\deg\sigma=2m-2.$ If $u\in \mk P_{n}[x,y]$ is a solution to (\ref{1}) for data $\sigma,$ then $\widetilde{u}$ is a solution to (\ref{1}) with data $M_{xy}^{n-m-2}\widetilde{\sigma}.$
\end{lem}
\begin{proof}
Let $v=\widetilde{u}.$ Then $v(x,y)=(xy)^{n}u(x^{-1},y^{-1}).$ Then
\begin{align*}
v_{x} &=nx^{n-1}y^{n}u(x^{-1},y^{-1})-x^{n-2}y^{n}u_{x}(x^{-1},y^{-1})\\
v_{y} &=nx^{n}y^{n-1}u(x^{-1},y^{-1})-x^{n}y^{n-2}u_{y}(x^{-1},y^{-1})\\
v_{xy} &=n^{2}x^{n-1}y^{n-1}u(x^{-1},y^{-1})-nx^{n-1}y^{n-2}u_{y}(x^{-1},y^{-1})\\
           &-nx^{n-2}y^{n-1}u_{x}(x^{-1},y^{-1})+x^{n-2}y^{n-2}u_{xy}(x^{-1},y^{-1}).
\end{align*}
This implies that
\begin{align*}
vv_{xy}-v_{x}v_{y}&=(xy)^{2n-2}(u(x^{-1},y^{-1})u_{xy}(x^{-1},y^{-1})-u_{x}(x^{-1},y^{-1})u_{y}(x^{-1},y^{-1}))\\
&=(xy)^{2n-2}\sigma(x^{-1},y^{-1}) u(x^{-1},y^{-1})\\
&=(xy)^{n-m-2}(xy)^{m}\sigma(x^{-1},y^{-1})\cdot (xy)^{n}u(x^{-1},y^{-1})\\
&=M_{xy}^{n-m-2}\widetilde{\sigma}(x,y)v(x,y)
\end{align*}
This proves our assertion.
\end{proof}

This lemma leads to:

\begin{cor}
Let $\sigma\in \mk P_{m}[x,y]$ be given with $\deg \sigma=2m-2$. If $u\in \mk P_{m+2}[x,y]$ is a solution to (\ref{1}) for data $\sigma,$ then $\widetilde{u}$ is a solution to (\ref{1}) with data $\widetilde{\sigma}.$
\end{cor}

Apparently, it is not simple to determine whether the set $V_{\sigma}$ is empty or not. For the main purpose of this paper, we give only a partial solution to this question. \\

A polynomial $u$ in $\mk P_{\infty}[x,y]$ is called diagonal if $u=\mk p(A)$ for some diagonal matrix $A\in M_{\infty}(\mb C).$ If a polynomial $u$ is diagonal, we can find a polynomial $Q(t)\in\mb C[t]$ such that $u(x,y)=Q(xy).$ Here comes a natural question: given a diagonal polynomial $\sigma$ as a data of (\ref{1}), can we find a solution $u$ to (\ref{1}) such that $u$ is also diagonal. 
From now on, we only consider prime solutions to (\ref{1}).

\begin{thm}\label{thmd}
Let $\sigma\in\mk  P_{m}[x,y]$ be a diagonal polynomial of degree $2m-2$ with $\sigma(x,y)=S(xy)$ for some $S(t)\in\mb C[t].$ Then (\ref{1}) has a solution $u$ that is also diagonal if and only if there exists $N\in\mb N$ with $N\geq m+2$ such that the family of polynomial $\{f_{S}^{i}:N+1\leq i\leq 2N-1\}$ has a nonzero common root. Furthermore, if $N=m+2,$ then $u\in\mk  P_{m+2}[x,y]$ with $\deg u=2m+2.$
\end{thm}
\begin{proof}
Assume that $v(x,y)=q(xy)$ for some $q\in \mb C[t].$ Then
\begin{align}\label{opde}
vv_{xy}-v_{x}v_{y}-\sigma v&=(xy)q''(xy)q(xy)+q'(xy)q(xy)\\
&-(xy)(q'(xy))^{2}-S(xy)q(xy)\notag.
\end{align}
If $u(x,y)$ is a diagonal polynomial that solves (\ref{1}) for data $\sigma,$ and if we write $u(x,y)=Q(xy)$ for some $Q(t)\in\mb C[t],$ then by (\ref{opde}), $Q(t)$ solves (\ref{ode1}) for data $S(t)$ with $t=xy.$ Since $Q(t)$ is a polynomial solution to (\ref{ode1}) with data $S(t),$ Proposition \ref{propf} implies the result.

Let us prove the converse. Let $a$ be a nonzero common root of $\{f_{S}^{i}:N+1\leq i\leq 2N-1\}.$ Define $q_{i}$ by $q_{0}=1/a$ and $q_{i}=f_{S}^{i}(a)$ for $i\geq 1.$ By Proposition \ref{propf}, the polynomial $Q(t)=\sum_{i=0}^{\infty}q_{i}t^{i}$ solves (\ref{ode1}) with data $S(t).$ Define $u(x,y)=Q(xy).$ Then $u(x,y)$ is a polynomial. By (\ref{opde}), $u$ solves (\ref{1}) for data $\sigma.$ The rest follows from Corollary \ref{cor1}.
\end{proof}

This theorem enables us to find a class of polynomials $\sigma$ in $\mk P_{\infty}[x,y]$ such that $V_{\sigma}$ is nonempty. It would be interesting to find  criterions to know when $V_{\sigma}$ is nonempty for any $\sigma\in \mk P_{\infty}[x,y].$

\section{An explicit construction of a solution to the mean field equation for hyperelliptic curves}\label{smfe}

Let $H=(h_{ij})_{i,j=1}^{g}$ be a $g\times g$ positive definite hermitian matrix and consider the corresponding canonical metric $ds_{H}^{2}$ on the hyperelliptic curve $X$ of genus $g$ defined in the introduction. If we let $\sigma_{H}(x,y)$ be the complex polynomial $\sigma_{H}(x,y)=\sum_{i,j=1}^{g}h_{ij}x^{i-1}y^{j-1},$ then the canonical metric $ds_{H}^{2}$ on $X$ has the local expression
$$
ds_{H}^{2}=
\begin{cases}
\dsp\frac{\sigma_{H}(x,\bar{x})}{|y^{2}|}dx\otimes d\bar{x}&\mbox{ on $C_{0},$}\\
\dsp\frac{\widetilde{\sigma}_{H}(z,\bar{z})}{|w^{2}|}dz\otimes d\bar{z}&\mbox{ on $C_{0}'.$}
\end{cases}
$$

\begin{thm}\label{pmfe}
Suppose (\ref{1}) has a solution $u=\mk p(A)\in P_{g+1}[x,y]$ for the data $\sigma_{H}$ with $A\in M_{g+1}(\mb C)$ being positive definite. Then the function
$$
\varphi=
\begin{cases}
\dsp\frac{4|f(x)|}{u(x,\bar{x})} &\mbox{on $C_{0},$}\\
\dsp\frac{4|g(z)|}{\widetilde{u}(z,\bar{z})} &\mbox{on $C_{0}'.$}
\end{cases}
$$
is a globally defined nonnegative smooth function whose zero set coincides with the set of Weierstrass points of $X$ and $\psi=\log\varphi$ defines smooth function on $X\setminus\{P_{1},\cdots,P_{2g+2}\}$ satisfying (\ref{MFE})
\end{thm}
\begin{proof}
The proof is the same as that given in our previous paper; we give a sketch of the proof. For more details, see \cite{FL}. Let us verify that $\Delta\psi+e^{\psi}=0$ on $U=X\setminus\{P_{1},\cdots,P_{2g+2}\}.$ We will prove this equation on $U\cap C_{0}.$ Since $u$ satisfies (\ref{1}), on $U\cap C_{0},$
$$\frac{\partial^{2}}{\partial x\partial \bar{x}}\log\varphi=-\frac{u_{x\bar{x}}u-u_{x}u_{\bar{x}}}{u^{2}}=-\frac{\sigma u}{u^{2}}=-\frac{\sigma}{u}.$$
As a consequence,
$$\Delta_{H}\psi=4\frac{|f(x)|}{\sigma(x,\bar{x})}\frac{\partial^{2}}{\partial x\partial \bar{x}}\log\varphi=-4\frac{|f(x)|}{u(x,\bar{x})}=-\varphi=-e^{\psi}.$$
Similarly, the equation holds on $U\cap C_{0}'.$

Let $P=P_{k}$ be a Weierstrass point of $X.$ In a coordinate neighborhood $(U_{P},\zeta)$ of $P=P_{k},$ where $\zeta=\sqrt{x-e_{k}},$ the function $\psi$ has a local expression $\psi=2\log|\zeta|+\alpha,$ where $\alpha$ is nonzero smooth function on $U_{P}.$ By classical analysis, the action of the Laplace operator $\Delta$ on $\psi$ creates a Dirac delta measure $4\pi\delta_{P_{k}}.$ We complete the proof of our assertion.
\end{proof}



Since $H$ is $g\times g$ positive definite hermitian matrix, there exits a $g\times g$ unitary matrix $U$ such that $U^{*}HU$ is a diagonal matrix $\Lambda$ with positive diagonals. We assume that $\Lambda=\diag(\lambda_{1},\cdots,\lambda_{g})$ with $\lambda_{i}>0$ for $1\leq i\leq g.$ Let us denote $S_{\Lambda}(t)=\sum_{i=0}^{g-1}\lambda_{i+1}t^{i},$ then the polynomial $\sigma_{\Lambda}(x,y)=S_{\Lambda}(xy)$ is diagonal. In other words, we consider the canonical metric on $X$ of the form
$$
ds_{\Lambda}^{2}=
\begin{cases}
\dsp\frac{\sum_{i=1}^{g}\lambda_{i}(x\bar{x})^{i-1}}{|y^{2}|}dx\otimes d\bar{x}&\mbox{ on $C_{0},$}\\
\dsp\frac{\sum_{i=1}^{g}\lambda_{i}(z\bar{z})^{g-i}}{|w^{2}|}dz\otimes d\bar{z}&\mbox{ on $C_{0}'.$}
\end{cases}
$$
One can use Theorem \ref{thmd} to determine diagonal solutions to (\ref{1}) for $\sigma_{H}$ in this case and to obtain ``positive definite'' solutions to (\ref{1}) for $\sigma_{H},$ we need further analysis, i.e. solutions $u=\mk p(A)$ so that $A$ is a $g\times g$ positive definite hermitian matrix. For $g\geq 2,$ let $\{F^{i}(x_{0},\cdots,x_{g-1},t):i\geq 1\}$ be the sequence of polynomials defined in (\ref{poly}). Let $V$ be the affine algebraic subset of $\mb C^{g+1}$ defined by the zero set of the polynomials $\{F^{g+2},\cdots,F^{2g-1}\}$ and $D_{+}^{g+1}$ be the set of all $n$-tuples of real numbers $(a_{1},\cdots,a_{g+1})$ such that $a_{i}>0$ for all $1\leq i\leq g+1$ and $Q_{+}^{g+1}$ be the subset of all $D_{+}^{g+1}$ consisting of points $(a_{0},\cdots,a_{g+1})$ so that $F^{i}(a_{0},\cdots,a_{g+1})>0$ for $1\leq i\leq g+1.$ If there exists a positive real number $a$ such that $(\Lambda,a)\in V\cap Q_{+}^{g+1},$ then the polynomial

\begin{equation}
u_{(\Lambda,a)}(x,y)=\frac{1}{a}+\sum_{i=1}^{g+1}F^{i}(\Lambda,a)(xy)^{i}
\label{determining function for solution}
\end{equation}

\noindent solves for (\ref{1}) and equals $\mk p(A)$ for $A=\diag(1/a,F^{1}(\Lambda,a),\cdots,F^{g+1}(\Lambda,a))$ and hence determines a solution to (\ref{MFE}) by
\begin{equation}
\psi_{(\Lambda,a)}
=\begin{cases}
\dsp\log\frac{|f(x)|}{u_{(\Lambda,a)}(x,\bar{x})} &\mbox{on $C_{0},$}\\
\dsp\log\frac{|g(z)|}{\widetilde{u}_{(\Lambda,a)}(z,\bar{z})} &\mbox{on $C_{0}',$}
\end{cases}
\label{exlicit solution}
\end{equation}
for $\Lambda=\diag(\lambda_{1},\cdots,\lambda_{g}).$ Let us take a look at the case when $X$ is of genus two and of genus three.


\begin{ex}\label{g2}
Let $X$ be the hyperelliptic curve defined by the equation $y^{2}=f(x)$ with metric $ds^{2}$ where $f(x)$ is a degree $6$ polynomial with $6$ distinct root and
$$ds^{2}=\frac{1+|x|^{2}}{|y^{2}|}dx\otimes d\bar{x}.$$
In this case, $S(t)=1+t.$ 
Then $f_{S}^{1}(t)=1$ and $f_{S}^{2}(t)=(t+1)/4$ and $f_{S}^{3}(t)=t/9$ and
\begin{align*}
f_{S}^{4}(t)&=-\frac{1}{192}t^{2}+\frac{1}{64}t\\
f_{S}^{5}(t)&=\frac{1}{1800}t^{3}-\frac{1}{600}t^{2}.
\end{align*}
One sees that $3$ is the common root of the polynomials $f_{S}^{4}(t)$ and $f_{S}^{5}(t).$ Then $h_{0}=1/3$ and $h_{1}=f_{S}^{1}(3)=1$ and $h_{2}=f_{S}^{2}(3)=1$ and $h_{3}=f_{S}^{3}(3)=1/3.$
We obtain a polynomial $u(x,y)$ by
$$u(x,y)=\frac{1}{3}+xy+(xy)^{2}+\frac{1}{3}(xy)^{3}$$
which solves (\ref{1}) for data $\sigma(x,y)=1+xy.$ 
This gives us the solution $\psi$ to (\ref{MFE}) by the construction of Theorem (\ref{pmfe}) for the genus $2$ hyperelliptic curve:
$$
\psi=
\begin{cases}
\dsp\log\frac{12|f(x)|}{(1+|x|^{2})^{3}} &\mbox{on $C_{0},$}\\
\dsp\log\frac{12|g(z)|}{(1+|z|^{2})^{3}} &\mbox{on $C_{0}',$}
\end{cases}
$$
The result coincides with that obtained in our previous paper.
\end{ex}

\begin{ex}
Let $X$ be the hyperelliptic curve defined by the equation $y^{2}=f(x)$ with the metric $ds^{2}$ where $f(x)$ is a polynomial with $8$ distinct roots and
$$ds^{2}=\frac{1+|x|^{2}+|x|^{4}}{|y^{2}|}dx\otimes d\bar{x}.$$
In this case, $S(t)=1+t+t^{2}.$ Then
$f_{S}^{1}(t)=1$ and $f_{S}^{2}(t)=(t+1)/4$ and $f_{S}^{3}(t)=(t+1)/9$ and $f_{S}^{4}(t)=(-t^{2}+11t)/192$ and
\begin{align*}
f_{S}^{5}(t) &=\frac{1}{1800}t^{3}-\frac{11}{1800}t^{2}+\frac{1}{75}t\\
f_{S}^{6}(t) &=-\frac{1}{11520}t^{4}+\frac{11}{11520}t^{3}-\frac{1}{405}t^{2}+\frac{1}{324}t\\
f_{S}^{7}(t)&=\frac{1}{58800}t^{5}-\frac{401}{2116800}t^{4}+\frac{373}{705600}t^{3}-\frac{43}{52920}t^{2}
\end{align*}
One sees that $8$ is the common root of the polynomials $f_{S}^{5}(t)$ and $f_{S}^{6}(t)$ and $f_{S}^{7}(t).$ We see that $h_{0}=1/8$ and $h_{1}=f_{S}^{1}(8)=1$ and $h_{2}=f_{S}^{2}(8)=9/4$ and $h_{3}=f_{S}^{3}(8)=1$ and $h_{4}=f_{S}^{4}(8)=1/8.$ We obtain a polynomial
$$u(x,y)=\frac{1}{8}+xy+\frac{9}{4}(xy)^{2}+(xy)^{3}+\frac{1}{8}(xy)^{4}$$
that solves (\ref{1}) for the data $\sigma(x,y)=1+xy+(xy)^{2}.$ 
This gives us the solution $\psi$ to (\ref{MFE}) by Theorem \ref{pmfe} for the genus $3$ hyperelliptic curve:
$$
\psi=
\begin{cases}
\dsp\log\frac{12|f(x)|}{\left(\frac{1}{8}+|x|^{2}+\frac{9}{4}|x|^{4}+|x|^{6}+\frac{1}{8}|x|^{8}\right)} &\mbox{on $C_{0},$}\\
\dsp\log\frac{12|g(z)|}{\left(\frac{1}{8}+|z|^{2}+\frac{9}{4}|z|^{4}+|z|^{6}+\frac{1}{8}|z|^{8}\right)} &\mbox{on $C_{0}',$}
\end{cases}
$$

\end{ex}

\section{adiabatic limit of solutions to mean field equations}\label{al}

We propose a possible direction following the results above. Rescale the canonical metric by $\gamma \in \mathbb{R}^+,$ i.e. we consider the rescaling of the canonical metric $ds_{\Lambda,\gamma}^{2}=\gamma ds_{\Lambda}^{2}.$  With respect to this metric, the mean field equation is equivalent to

\begin{equation}\label{MFE rescaled}
\Delta\psi_\gamma+\gamma e^{\psi_\gamma}=4\pi \gamma \sum_{i=1}^{2g+2}\delta_{P_{i}}
\end{equation}

\noindent with respect to $ds_\Lambda^2$\footnote{For convenience, we use $\Delta$ instead of $\Delta_{\Lambda}$ in this section.}. Following from the analysis in \cite{KZ}, we study the existence of solution to this equation for small $\gamma$, as well as the limit of the solutions $\{\psi_{\gamma}\}$ as $\gamma \to 0$. Directly observing \eqref{MFE rescaled}, we naturally expect $\Delta\psi_\gamma \to 0$ as $\gamma \to 0$, or that $\psi_\gamma$ approaches to a constant function since $X$ is a connected closed manifold. Classical analysis from \cite{KZ} confirms both expectations. We normalize the metrics so that the area of $X$ is $1$. Let $W^{k,p}(X)$ be the completion of $C^{\infty}(X)$ with respect to the $(k,p)$-norm:
$$\|u\|_{W^{k,p}(X)}=\sum_{j=0}^{k}\left(\int_{X}|\nabla^{j}u|^{p}d\nu\right)^{1/p},$$
where $\nabla^{j}u$ is the $j$ th covariant derivative derivative of $u.$ We call $W^{k,p}(X)$ the Sobolev $(k,p)$-space on $X.$\footnote{In some context, people use $H^{k,p}(X)$ for Sobolev $(k,p)$ spaces.} A technical analytic statement is needed to conclude the asymptotic behaviors:
\begin{prop}
If $u_j \to u$ weakly in $W^{1,2}(X)$, then $e^{u_j} \to e^u$ strongly in $L^2(X)$.
\label{embedding prop}
\end{prop}
\begin{proof}
For the proof, see (3.7) in \cite{KZ}
\end{proof}

\begin{thm}[Adiabatic Limit]\label{adiabatic limit}
Solution to \eqref{MFE rescaled} exists for all $\gamma$ small enough and approaches a constant in $W^{2,2}(X)$ as $\gamma \to 0$.
\end{thm}
\begin{proof}
We only sketch the existence part of the proof since it is a replica of the proof from Theorem 7.2 in \cite{KZ}.
Let

  \begin{equation}
  \psi_\gamma := v_\gamma +4\pi \gamma \sum_{i=1}^{2g+2} G_i,
  \label{substituion to KW}
  \end{equation}

 \noindent where $G_i$ is the Green's function satisfying $\Delta G_i = - \delta_{P_i} +1$. Solving \eqref{MFE rescaled} is then equivalent to solving the following equation

  \begin{equation}
  \Delta v_\gamma + \gamma h e^{v_\gamma} = 8\pi\gamma (g+1),
  \label{Kazdan Warner rescaled}
  \end{equation}

  \noindent where the function $h=\exp\left(4\pi\sum_{i=1}^{2g+2} G_i\right)\in C^\infty(X)$ is nonnegative with zero set precisely the Weierstrass points. This is a Kazdan-Warner equation of the type discussed in section 7 from \cite{KZ}, which is solved by variational method. One notes that \eqref{Kazdan Warner rescaled} is the minimizing equation to the functional

  \begin{equation}
  J(u) = \int_X \left(\frac{1}{2} |\nabla u|^2 + 8\pi\gamma (g+1) u\right) d\nu
  \label{minimizing functional}
  \end{equation}

  \noindent on the subset $B \subset W^{1,2}(X)$ satisfying the constraint equation

  \begin{equation}
  \int_X he^{u} d\nu = 8\pi(g+1).
  \label{constraint equation}
  \end{equation}

  Following identical reasoning, we have the following estimate for $J$:

  \begin{equation}
    J(u) \geq \frac{1}{4\beta}(2\beta - 8\pi\gamma(g+1))\|\nabla u\|_{L^{2}(X)}^2+\delta,
    \label{lower bound for J}
  \end{equation}

\noindent where $\delta$ is a constant and $\beta$ is a Trudinger constant for $X$ both independent of $\gamma$. More precisely, $\beta$ is a positive constant so that

\[\int_X e^{\beta v^2} d\nu\]

\noindent are uniformly bounded for all $v\in W^{1,2}(X)$ with $\bar{v}=0$ and $\|\nabla v\|_{L^{2}(X)} \leq 1$. Such a constant always exists for surfaces (cf. (3.4) in \cite{KZ}). Therefore, for $\gamma$ small enough so that $2\beta - 8\pi\gamma(g+1)>0$, $J$ is bounded below and positive.

For each $\gamma$, \eqref{lower bound for J} and Sobolev embedding shows that the minimizing sequence $\{v_\gamma^i\}$ of $J$ is contained in a fixed ball of radius $R_\gamma$ in $W^{1,2}(X)$, which is weakly compact. Passing to a subsequence, let $v_\gamma$ be the weak limit. Arguments in the proof of Theorem 5.3 in \cite{KZ} show that $v_\gamma$ minimizes $J$ in $B$ and therefore is a strong limit and solution to \eqref{Kazdan Warner rescaled}. The proof there also provides a regularity argument, which is applicable to our case here, to show that $v_\gamma$ is actually smooth. The existence of a smooth solution for each $\gamma$ is established.

Furthermore, one notices that the radii $R_\gamma$ are uniformly controlled over $\gamma$ (in fact proportional to $(2\beta - 8\pi\gamma(g+1))^{-1}$) and therefore $\{v_\gamma\}$ are uniformly bounded in $W^{1,2}(X)$. Following identical arguments, let $v $ be the limit of $v_\gamma$ in $W^{1,2}(X)$. Proposition \ref{embedding prop} then implies that $e^{v_\gamma}$ converge to $e^v$ in $L^2(X)$ and therefore are uniformly bounded in $L^2(X)$. It then follows from elliptic regularity of $\Delta$ in \eqref{Kazdan Warner rescaled}:

\begin{equation}
  \|v_\gamma\|_{W^{2,2}(X)} \leq c(\gamma \|8\pi(g+1)-he^{v_\gamma}\|_{L^{2}(X)}+\|v_\gamma\|_{L^2(X)})
  \label{elliptic bound}
\end{equation}

\noindent that $v_\gamma$ are uniformly bounded in $W^{2,2}(X)$. The estimate, together with some Schauder estimates, also imply that $v\in C^\infty(X)$. After taking a subsequence, we conclude that $v_\gamma \to v$ in $W^{2,2}(X)$. Taking the limit $\gamma \to 0$ in \eqref{Kazdan Warner rescaled}, it then follows that

\begin{equation}
\Delta v = \lim_{\gamma \to 0} \Delta v_\gamma =0
\label{limit of Laplacian}
\end{equation}

\noindent and therefore $v$ is a constant function since $X$ is closed.

\end{proof}

It is of great interest, as stated in \cite{KZ}, to study the upper bound of $\gamma$:

\[\gamma_0=\frac{\beta}{4\pi(g+1)}\]

\noindent for \eqref{Kazdan Warner rescaled} to be solvable, a quantity related to the geometry of $X$. It is not immediately clear wether $\gamma_0 \geq 1$, despite the explicit solution to \eqref{Kazdan Warner rescaled} with $\gamma=1$ in section \ref{smfe}. One may attempts to construct a variation of \eqref{exlicit solution} depending on $\gamma$, and its corresponding mean field equation so that the limiting solution at $\gamma=0$ coincides with that of Theorem \ref{adiabatic limit}. Such a conjecture provides significant geometric insight. In the case of Example \ref{g2} where solutions are precisely the logarithm of Gaussian curvatures, the limiting solution suggests that the manifold deforms into $\mathbb{S}^2$, a sign of topological jumps, or bubbling.



\bibliography{sampartb}

\end{document}